\documentclass[reqno]{amsart} 

\usepackage[foot]{amsaddr}

\usepackage[ruled,lined,norelsize]{algorithm2e}

\usepackage{times}
\usepackage{url}
\usepackage{amsmath,amssymb,amsfonts,amsthm}
\usepackage{stmaryrd}

\usepackage{xcolor}
\usepackage{mathtools}
\usepackage{microtype}
\usepackage{enumerate}

\newtheorem{theorem}{Theorem}
\newtheorem{proposition}[theorem]{Proposition}

\newcommand{\R}{{\mathbb{R}}}
\newcommand{\mF}{{\mathcal{F}}}
\newcommand{\mX}{{\mathcal{X}}}
\newcommand{\mY}{{\mathcal{Y}}}

\newcommand{\bx}{{\mathbf{x}}}
\newcommand{\by}{{\mathbf{y}}}

\newcommand{\hopmx}{x_{\text{HOPM}}}
\newcommand{\alsx}{x_{\text{ALS}}}
\newcommand{\bA}{{\mathbf{A}}}
\newcommand{\bB}{{\mathbf{B}}}
\newcommand{\bC}{{\mathbf{C}}}

\newcommand{\tenprod}{{\circ}}
\newcommand{\frob}{{\mathsf{F}}}
\newcommand{\abs}[1]{{\lvert #1 \rvert}}

\DeclareMathOperator{\spann}{span}
\DeclareMathOperator{\argmin}{argmin}

\begin{document}

\author[A. Uschmajew]{Andr\'e Uschmajew$^\ast$}
\address{\footnotesize\upshape${}^\ast$MATHICSE-ANCHP, Section de Math\'ematiques, \'Ecole Polytechnique F\'ed\'erale de Lausanne, 1015 Lausanne, Switzerland. Current address: Hausdorff Center for Mathematics \& Institute for Numerical Simulation, University of Bonn, 53115 Bonn, Germany (\texttt{uschmajew@ins.uni-bonn.de}).}

\title[Convergence of the higher-order power method]{A new convergence proof for the higher-order power method and generalizations}

\begin{abstract}
A proof for the point-wise convergence of the factors in the higher-order power method for tensors towards a critical point is given. It is obtained by applying established results from the theory of \L{}ojasiewicz inequalities to the equivalent, unconstrained alternating least squares algorithm for best rank-one tensor approximation. 
\end{abstract}



\maketitle

\section{Introduction}

Finding the best rank-one approximation to a given higher-order tensor is equivalent to finding its largest tensor singular value (also known as its spectral norm), which is defined as the maximum of the associated multilinear form on the product of unit spheres. This simplest of all low-rank tensor approximation tasks is of large interest in its own, but also constitutes the main building-block when constructing approximations of higher rank by means of rank-one updates, see for example the references given in~\cite[Sec.~3.3]{GrasedyckKressnerTobler2013}.


The \emph{higher-order power method} (HOPM)~\cite{DeLathauweretal1995,DeLathauweretal2000} is a simple, effective, and widely used optimization algorithm to approximately solve the task. The name comes from the fact that it is the straight-forward generalization of an alternating power method for finding a pair of dominant left and right singular vectors of a matrix. Depending on the scaling strategy used for the iterates during the process, the higher-order power method can be seen as an \emph{alternating least squares} (ALS) algorithm, see~\cite{KoldaBader2009} and references therein.

Despite its importance, a satisfactory convergence theory for the HOPM was missing until recently. Clearly, the convergence of the generated sequence of approximated singular values follows easily from the monotonicity of the method~\cite{RegaliaKofidis2000}. More interesting and important, however, is the question of single-point (and not just sub-sequential) convergence of the sequences of generated rank-one tensors or even their factors to a critical point of the problem. The local convergence for starting guesses close enough to a critical point was established in~\cite{GolubZhang2001} and~\cite{Uschmajew2012}, but the made assumptions remained somewhat restrictive. Concerning global convergence, the investigations of Mohlenkamp~\cite{Mohlenkamp2013} showed that the sequence of rank-one tensors generated by ALS is bounded, and that their consecutive differences are absolutely square summable and hence converge to zero. This would imply convergence of the method, if the set of cluster points, each of which must be a critical point, contains at least one isolated point which then is the limit. In a recent work by Wang and Chu~\cite{WangChu2014} this last issue was addressed by arguing that 
for almost every tensor the second-order derivative at zeros of the projected gradient of the cost function is regular, and hence critical points isolated. In this way, global convergence of the higher-order power method has been established, at least for almost every tensor.

The outlined argumentation appears, however, somewhat intricate. In this paper, we propose an alternative convergence proof based on an elegant method from the theory of analytical gradient flows, whose foundation is the \emph{\L{}ojasiewicz gradient inequality} -- a powerful feature of real-analytic functions. Simply speaking, the validity of this inequality at a cluster point of a gradient-related descent iteration enforces absolute summability of increments, which implies convergence~\cite{AbsilMahonyAndrews2005}. The continuous counterpart of this methodology is mentioned in~\cite{WangChu2014}, but the possibility to directly apply the available results on discrete gradient flows to ALS was not explored. This is what we shall do in the present paper.

In~\cite{XuYin2013}, Xu and Yin used a further generalization, the Kurdyka-\L{}ojasiewicz inequality, to obtain convergence results for a variety of cyclic block coordinate descent methods when applied to a large class of strongly block multiconvex functions. This includes a wide range of alternating block techniques for regularized low-rank tensor optimization tasks. In principle, our considerations will show that even \emph{without} regularization, the ALS algorithm for rank-one approximation is a member of this problem class. The key observation is an insight gained in~\cite{LiUschmajewZhang2013}, that the norms of the factors generated by ALS remain bounded from above and below, even when no normalization is used. In particular, norm constraints can be avoided in the analysis for this reason.

The focus on one specific method allows us to present the logic of the convergence proof in a simplified form compared to the very general reasoning in~\cite{XuYin2013}. As a result, we obtain the global convergence of the higher-order power method as the last link in a transparent chain of simple arguments. Admittedly, the abstract results based on the \L{}ojasiewicz gradient inequality, that are invoked at one point in the presentation, constitute a nontrivial ingredient in our proof, but they can be regarded as well-established by now.

The paper is organized as follows. In Sec.~\ref{sec: best rank-one approximation} we introduce the notation used, define the higher-order power method, and the equivalent alternating least squares algorithm. In Sec.~\ref{sec:convergence of ALS} we state the abstract convergence results from the literature on which we rely, and then prove that they can be applied to rank-one ALS. The main result is Theorem~\ref{th:main result}. Finally, Sec.~\ref{sec: generalization} is devoted to generalizations of the used arguments to strongly convex optimization tasks in other multilinear tensor formats by means of ALS-type algorithms~\cite{BeylkinMohlenkamp2005,HoltzRohwedderSchneider2012,KoldaBader2009,RohwedderUschmajew2013,OseledetsDolgov2012,Tobler2012}. We explain why for formats other than rank-one, regularization is typically unavoidable to achieve similar strong results.

\section{Best rank-one approximation}\label{sec: best rank-one approximation}

In this section, we recall the higher-order power method and the alternating least squares algorithm, and explain their connection in more detail.

\subsection{Preliminaries}
Let $d \ge 3$ and $n_1,n_2,\dots,n_d \in \mathbb{N}$ be given. The elements of the Cartesian product $\R^{n_1} \times \R^{n_2} \times \dots \times \R^{n_d}$ will be either explicitly denoted by tuples $(x^1,x^2,\dots,x^d)$, or abbreviated by
\[
\bx = (x^1,x^2,\dots,x^d).
\]
The elements in $\R^{n_1 \times n_2 \times \dots \times n_d}$ will be called \emph{tensors} and are treated as multi-dimensional arrays with entries labeled by multi-indices $i_1,i,\dots,i_d$. For tensors we use $\langle \cdot, \cdot \rangle_\frob$ and $\| \cdot \|_\frob$ to denote the Frobenius inner product and norm, respectively. For vectors, we omit the subscript $\frob$ when denoting the Euclidean inner product and norm. Similarly, the norm of a tuple $\bx$ will be denoted by $\| \bx \| = (\| x^1 \|^2 + \| x^2 \|^2 + \dots + \| x^d \|^2)^{1/2}$.

Consider the multilinear map
\(
\tau_1 \vcentcolon \R^{n_1} \times \R^{n_2} \times \dots \times \R^{n_d} \to \R^{n_1 \times n_2 \times \dots \times n_d}
\)
defined by
\begin{equation}\label{eq:rank-one tensors}
 \tau_1(\bx) = x^1 \tenprod x^2 \tenprod \cdots \tenprod x^d,
\end{equation}
where $\circ$ is the outer product. It means that $[\tau_1(\bx)]_{i_1,i_2,\dots,i_d} = x^1_{i_1} x^2_{i_2} \cdots x^d_{i_d}$. The non-zero tensors in the range of $\tau_1$ are called \emph{rank-one tensors}. The vectors $x^\mu$ will be called \emph{factors} of $\tau_1(\bx)$. A crucial property of $\tau_1$ is
\begin{equation}\label{eq:product of inner products}
\langle \tau_1(\bx), \tau_1(\by) \rangle_\frob = \langle x^1,y^1 \rangle \langle x^2,y^2\rangle \cdots \langle x^d,y^d \rangle,
\end{equation}
and therefore
\begin{equation}\label{eq: product of norms}
\| \tau_1(\bx) \|_\frob = \| x^1 \|^{} \|x^2 \|^{} \cdots \| x^d \|^{}.
\end{equation}

To a tensor $\mF \in \R^{n_1 \times n_2 \times \dots \times n_d}$ we associate a multilinear form $F$ defined as
\begin{align*}
F(\bx) = \langle \mF, \tau_1(\bx) \rangle_\frob 
= \sum_{i_1 = 1}^{n_1} \sum_{i_2 = 1}^{n_2} \cdots \sum_{i_d = 1}^{n_d} \mathcal{F}_{i_1,i_2,\dots,i_d}^{} x^1_{i_1} x^2_{i_2} \cdots x^d_{i_d}.
\end{align*}
For $\mu = 1,2,\dots,d$, we also define partial contractions $F^\mu(\bx)$ which are the vectors in $\R^{n_\mu}$, whose $i_\mu$th entry is
\[
\sum_{i_1 = 1}^{n_1} \cdots \sum_{i_{\mu-1} = 1}^{n_{\mu-1}} \sum_{i_{\mu+1} = 1}^{n_{\mu+1}} \cdots \sum_{i_d = 1}^{n_d} \mathcal{F}_{i_1,\dots,i_{\mu-1},i_\mu,i_{\mu+1},\dots,i_d}^{} x^1_{i_1} \cdots x^{\mu-1}_{i_{\mu-1}} x^{\mu+1}_{i_{\mu+1}} \cdots x^d_{i_d},
\]
that is, the contraction with $x^{\mu}$ is omitted. Equivalently, $F^\mu(\bx)$ may be defined as the unique vector in $\R^{n_\mu}$ satisfying
\begin{equation}\label{eq: mu contraction}
\langle F^\mu(\bx), x^\mu \rangle = F(\bx)
\end{equation}
for all $x^\mu$.

The algorithms we consider produce sequences iterates $(x_k^\mu)_k^{}$ for every component $\mu = 1,2,\dots,d$. We hence introduce the notation
\[
\bx_k^\mu = (x^1_k,\dots,x^\mu_k,x^{\mu+1}_{k-1},\dots,x^d_{k-1}).
\]
For convenience, let further $\bx_{k+1}^0 = \bx_k^{}$.

\subsection{Higher-order power method}
The critical values of $\bx \mapsto F(\bx) / (\|x^1\| \|x^2\| \cdots \|x^d\|)$ are called the \emph{singular values} of the tensor $\mF$~\cite{Lim2005}. The \emph{maximum singular value} is
\begin{equation}\label{eq:HOPM function}
\lambda^* = \max_{\|x^1\|^{} = \|x^2\|^{} = \dots = \| x^d \|^{} =1} F(\bx).
\end{equation}
This expression defines a norm  (the usual norm of a multilinear form), and so $\lambda^*$ is sometimes referred to as the \emph{spectral norm} of the underlying tensor $\mF$ in the literature. 
The higher-order power method (HOPM) is a cyclic block coordinate method to approximate $\lambda^*$. By~\eqref{eq: mu contraction}, the optimal choice for $x^{\mu}$ when fixing the other factors is
\begin{equation}\label{eq:HOPM update}
\hopmx^\mu = \frac{F^\mu(\bx)}{\| F^\mu(\bx) \|^{}}. 
\end{equation}
This already constitutes the HOPM method summarized as Algorithm~\ref{Alg:HOPM}. For clarity we use the notation $\by_k^\mu$ for the iterates of HOPM, and reserve $\bx_k^\mu$ for the iterates produced by the ALS algorithm (Algorithm~\ref{Alg:ALS}) introduced next.

\begin{algorithm}
\DontPrintSemicolon
\KwIn{Tensor $\mF \in \R^{n_1 \times n_2 \times \dots \times n_d}$, starting guess $\by_0$ with $F^1(\by_0) \neq 0$.}
$k \leftarrow 0$, $\lambda_0 = F(\by_0)$\;
\While{not converged}{
\For{$\mu = 1,2,\dots,d$}{
$\displaystyle y^\mu_{k+1} = \frac{F^\mu(\by_{k+1}^{\mu-1})}{\|F^\mu(\by_{k+1}^{\mu-1})\|^{}}$\;
}
$\lambda_{k+1} = F(\by_{k+1})$\;
$k \leftarrow k+1$\;
}
\caption{Higher-order power method (HOPM)}\label{Alg:HOPM}
\end{algorithm}

Note that since $F(\by^\mu_k)$ is not decreasing and $F(\by_1^1) > 0$, a division by zero will never occur.

\subsection{Alternating least squares}

Given $\mF \in \R^{n_1 \times n_2 \times \dots \times n_d}$, let
\begin{equation}\label{eq:considered function}
f(\bx) = \frac12\| \mF - \tau_1(\bx) \|_\frob^2.
\end{equation}
The best rank-one approximation problem consists in finding a minimizer for $f$. The corresponding block coordinate descent method is called alternating least squares (ALS). The name comes from the fact that the problem for a single factor $x^\mu$ with the others held fixed is a least squares problem with normal equation
\begin{equation}\label{eq: normal equation}
\begin{aligned}
0 &= \langle \mF - \tau_1(\bx), \tau_1(x^1,\dots,x^{\mu-1},y^\mu,x^{\mu+1}, \dots, x^d) \rangle_\frob \\
&=\Bigg\langle F^\mu(\bx) - \Bigg(\prod_{\nu \neq \mu} \| x^\nu \|^2 \Bigg) x^\mu, y^\mu \Bigg\rangle \quad \text{for all $y^\mu \in \R^{n_\mu}$,}
\end{aligned}
\end{equation}
where we have used~\eqref{eq:product of inner products} and~\eqref{eq: mu contraction}. Assuming $x^\nu \neq 0$ for all $\nu \neq \mu$, the unique solution is
\begin{equation}\label{eq: ALS update}
\alsx^\mu = \frac{F^\mu(\bx)}{\prod_{\nu \neq \mu} \| x^\nu \|^2}.
\end{equation}
The resulting ALS algorithm is noted as Algorithm~\ref{Alg:ALS}.
\begin{algorithm}
\DontPrintSemicolon
\KwIn{Tensor $\mF \in \R^{n_1 \times n_2 \times \dots \times n_d}$, starting guess $\bx_0$ with $F^1(\bx_0) \neq 0$.}$k \leftarrow 0$\;
\While{not converged}{
\For{$\mu = 1,2,\dots,d$}{
$\displaystyle x^\mu_{k+1} = \frac{F^\mu(\bx^{\mu-1}_{k+1})}{\| x^1_{k+1} \|^2 \cdots \| x^{\mu-1}_{k+1} \|^2 \cdot \| x^{\mu+1}_k \|^2 \cdots \| x^d_k \|^2 }$\;
}
$k \leftarrow k+1$\;
}
\caption{Alternating Least Squares (ALS)}\label{Alg:ALS}
\end{algorithm}

Note that $F^1(\bx_0) \neq 0$ implies $x_0^\mu \neq 0$ for $\mu = 2,\dots,d$, so the very first step of the algorithm is feasible and $x_1^1 \neq 0$. As we show in the next section, the sequences $\| x^\mu_k \|$ are monotonically increasing for every $\mu$, so the subsequent update steps also never fail.

In contrast to what is recommended in practice (see e.g.~\cite{KoldaBader2009}), our version of ALS omits any normalization of the factors $x^\mu_k$ during the process. This is by purpose, as it simplifies the analysis. From a theoretical viewpoint it also makes no difference, as any rescaling strategy does not affect the generated sequence of subspaces $\spann(x^\mu_k)$. In particular, comparing~\eqref{eq:HOPM update} and~\eqref{eq: ALS update} it is plain to see the equivalence of HOPM and ALS, but the detailed proof requires some notational effort.  
\begin{proposition}\label{prop: equivalence ALS HOPM}
Let $(\lambda_k,\by_k)$ and $(\bx_k)$ denote the iterates generated by Algorithms~\ref{Alg:HOPM} and~\ref{Alg:ALS}, respectively, when applied to the same starting guess $\by_0 = \bx_0$. Then it holds
\[
y_k^\mu = \frac{x_k^\mu}{\|x_k^\mu\|}, \quad \text{and} \quad \lambda_k = \| \tau_1(\bx_k) \|_\frob 
\]
for all $k \ge 1$ and all $\mu$. Also, if $\bx_*$ is a critical point of the function~\eqref{eq:considered function} with $\tau_1(\bx_*) \neq 0$, then $\by_*$ with $y^\mu_* = x^\mu_* / \| x^\mu_* \|$ is a critical point (w.r.t. the spherical constraints) of~\eqref{eq:HOPM function}.
\end{proposition}

\begin{proof}
 We show by induction that for every $k \ge 0$ and $\mu$ there exists $\alpha_k^\mu > 0$ such that $x_k^\mu = \alpha_k^\mu y^k_\mu$. For $k \ge 1$ this obviously implies $\alpha_k^\mu = \|x_k^\mu\|$, as $\| y^\mu_k\|=1$ by construction. We introduce an ordering of the pairs $(k,\mu)$ according to their appearance in the algorithms, i.e., $(k,\mu) > (\ell,\nu)$ if $k>\ell$ or if $k = \ell$ and $\mu > \nu$. Setting $\alpha_0^\mu = 1$ the assertion $y_0^\mu = \alpha_0^\mu x_0^\mu$ obviously holds for all pairs $(k,\mu)$ with $k=0$. Now fix $(k+1,\mu)$ and assume the relation has been proved for all previous pairs. Exploiting the multilinearity of $F^\mu(\bx)$ w.r.t. $x^1 , \dots, x^{\mu-1} , x^{\mu+1}, \dots, x^d$, and using $\alpha_{k+1}^\nu = \|x_{k+1}^\nu\|$ for $\nu < \mu$, it holds
 \begin{equation}\label{eq: explicit relation}
 \begin{aligned}
  x_{k+1}^\mu =  \frac{\alpha^{\mu+1}_k  \cdots \alpha^d_kF^\mu(\by^{\mu-1}_{k+1})}{\alpha^1_{k+1} \cdots  \alpha^{\mu-1}_{k+1} \| x^{\mu+1}_k \|^2 \cdots \| x^d_k \|^2 } 
  = \frac{\alpha^{\mu+1}_k  \cdots \alpha^d_k \| F(\by^{\mu-1}_{k+1}) \| }{\alpha^1_{k+1} \cdots  \alpha^{\mu-1}_{k+1} \| x^{\mu+1}_k \|^2 \cdots \| x^d_k \|^2 } \cdot y^\mu_{k+1}
 \end{aligned}
 \end{equation}
 (note that $\alpha^{\mu+1}_k  \cdots \alpha^d_k$ and $\| x^{\mu+1}_k \|\cdots \| x^d_k \|$ also cancel once $k \ge 1$). Hence $\alpha^\mu_{k+1}$ equals the fraction on the right side of~\eqref{eq: explicit relation}, which is positive.
 
 Now that we have proved $x^\mu_k = \alpha^\mu_k y^\mu_k$ with $\alpha^\mu_k = \| x^\mu_k\|$ for all $k \ge 1$,~\eqref{eq: explicit relation} and~\eqref{eq: product of norms} imply
 \[
 \|\tau_1(\bx_{k+1}^\mu) \|_\frob =  \|F^\mu(\by^{\mu-1}_{k+1}) \|.
 \]
 By definition of $y_{k+1}^{\mu}$ and~\eqref{eq: mu contraction},
 \[
 \| F^\mu(\by^{\mu-1}_{k+1}) \| = \langle F^\mu(\by^{\mu-1}_{k+1}), y_{k+1}^{\mu} \rangle = \langle F^\mu(\by^{\mu}_{k+1}), y_{k+1}^{\mu} \rangle = F(\by^{\mu}_{k+1}),
 \]
 where the second equality holds because $F^\mu(\bx)$ never depends on $x^\mu$. In summary,
 \begin{equation}\label{eq: norm of taukmu}
 \|\tau_1(\bx_{k+1}^\mu) \|_\frob = F(\by^{\mu}_{k+1}),
 \end{equation}
 and in particular $\|\tau_1(\bx_{k+1}) \|_\frob = F(\by_{k+1}) = \lambda_{k+1}$.
 
 Finally, let $\bx_*$ be a critical point of the function~\eqref{eq:considered function} with $\tau_1(\bx_*) \neq 0$. Then $\bx^*$ is a stationary point of~\eqref{eq: ALS update}, so that for all $\mu$ it holds $x^\mu_* = \alpha_*^\mu F^\mu(\bx_*)$ with $\alpha^\mu_* = \prod_{\nu \neq \mu} \| x^\nu_* \|^2 \neq 0$. Let $y_*^\mu = x_*^\mu / \| x_*^\mu \|$, then using multilinearity it also follows that
 \begin{equation}\label{eq: stationarity for y}
 \beta_*^\mu y^\mu_* = F^\mu(\by_*)
 \end{equation}
 for some $\beta_*^\mu \neq 0$. The tangent space to the unit sphere at $y_*^\mu$ consists of all vectors $\delta y^\mu_*$ orthogonal to $y_*^\mu$. Hence $\by_*$ is a critical point of $F$ on the Cartesian product of spheres, if for every $\mu$ it holds $\langle \nabla_\mu F(\by_*), \delta y^\mu_* \rangle = 0$ for all such $\delta y^\mu_*$. But since $F$ is linear with respect to every block variable, this is the case, as
 \[
 \langle \nabla_\mu F(\by_*), \delta y^\mu_* \rangle = \langle F^\mu(\by_*), \delta y^\mu_* \rangle = \beta_*^\mu \langle y_*^\mu , \delta y^\mu_* \rangle = 0
 \]
 by~\eqref{eq: mu contraction} and~\eqref{eq: stationarity for y}. 
\end{proof}

As a result, we can prove convergence of HOPM by proving convergence of ALS.

\section{Convergence of alternating least squares}\label{sec:convergence of ALS}

The global, point-wise convergence of the iterates generated by Algorithm~\ref{Alg:ALS} will be deduced from known results based on the \L{}ojasiewicz gradient inequality. We first recall the required abstract properties, and then show that they hold for Algorithm~\ref{Alg:ALS}.

\subsection{Point-wise convergence via \L{}ojasiewicz inequality}
Our aim is to apply the following result~\cite[Theorem~3.2]{AbsilMahonyAndrews2005}.

\begin{theorem}\label{th:convergence via Lojasiewicz}
Let $f \vcentcolon V \to \R$ be a real-analytic function on a finite-dimensional real vector space $V$, and let $(\bx_k) \subset \R^{n}$ be a sequence satisfying
\begin{equation}\label{eq: minimum decrease}
f(\bx_k) - f(\bx_{k+1}) \ge \sigma \| \nabla f(\bx_k) \|^{} \| \bx_{k+1} - \bx_k \|^{}
\end{equation}
for all large enough $k$ and some $\sigma > 0$. Assume further that the implication
\begin{equation}\label{eq: stationarity condition}
[f(\bx_{k+1}) = f(\bx_k)] \quad \Rightarrow \quad \bx_{k+1} = \bx_k
\end{equation}
holds. Then a cluster point $\bx_*$ of the sequence $(\bx_k)$ must be its limit. In particular, if the sequence is bounded, it is convergent.
\end{theorem}

The key ingredient in the proof of this theorem is the \L{}ojasiewicz gradient inequality,
\begin{equation}\label{eq:Lojasiewicz inequality}
\abs{f(\bx) - f(\bx_*)}^{1-\theta} \le \Lambda \| \nabla f(\bx) \|^{}, 
\end{equation}
which can be shown to hold in some (unknown) neighborhood of $\bx_*$ when $f$ is real-analytic~\cite[p.~92]{Lojasiewicz1965}. The constants $\Lambda > 0$ and $\theta \in (0,1/2]$ are typically not explicitly known as well. Yet, in combination with~\eqref{eq: minimum decrease} and~\eqref{eq: stationarity condition}, the \L{}ojasiewicz gradient inequality allows to prove that the norms $\| \bx_{k+1} - \bx_k \|$ of increments are summable, which proves convergence of the sequence $(\bx_k)$.

Under stronger conditions, one can conclude that the limit is a critical point of $f$. The following theorem will be applicable to Algorithm~\ref{Alg:ALS}, although the convergence rate estimates remain of minor use as long as no a-priori results on the expected value of the \L{}ojasiewicz exponent $\theta$ are available.

\begin{theorem}\label{th:convergence rate}
Under the conditions of Theorem~\ref{th:convergence via Lojasiewicz}, assume further that there exists $\kappa > 0$ such that
\begin{equation}\label{eq:minimum stepsize safeguard}
\| \bx_{k+1} - \bx_k \|^{} \ge \kappa \| \nabla f(\bx_k) \|
\end{equation}
for all large enough $k$. Then $\nabla f (\bx_*) = 0$, and the convergence rate can be estimated as follows:
\begin{equation}\label{eq:rate}
\| \bx_* - \bx_k \| \lesssim \begin{cases} q^{k} \quad &\text{if $\theta = \frac{1}{2}$ (for some $0 < q < 1$)}, \\ k^{- \frac{\theta }{1 -2 \theta }} \quad &\text{if $0<\theta<\frac{1}{2}$.}  \end{cases}
\end{equation}

\end{theorem}

We were not able to identify the original reference for this theorem which seems rather scattered through the literature, see e.g.~\cite{AttouchBolte2009,Levitt2012}. For concreteness we point to~\cite{SchneiderUschmajew2014}, where Theorems~\ref{th:convergence via Lojasiewicz} and~\ref{th:convergence rate} are proved in the stated form.

\subsection{Application to Algorithm~\ref{Alg:ALS}}\label{sec:application to ALS}

Let now $f$ be the function~\eqref{eq:considered function} again, and $(\bx_k)$ the sequence generated by Algorithm~\ref{Alg:ALS}. By a chain of simple arguments, we will show that the required properties~\eqref{eq: minimum decrease},~\eqref{eq: stationarity condition}, and~\eqref{eq:minimum stepsize safeguard} are satisfied. As $F^1(\bx_0)$ only depends on $x_0^2,\dots,x_0^d$, we assume now without loss in generality that $x_0^1 = x_1^1$ to avoid special treatment of the very first update in the following proofs.

The first two results are well-known, and express the monotonicity of the algorithms.

\begin{proposition}\label{prop:necessary optimality condition}
For all $k \ge 1$ and $\mu = 1,2,\dots,d$ it holds
\[
\| \mF \|_\frob^2 = \| \tau_1(\bx_k^\mu)\|^2_\frob + \| \mF - \tau_1 (\bx_k^\mu) \|_\frob^2.
\]
\end{proposition}
\begin{proof}
This is a necessary optimality condition for the least squares problem that was solved to obtain $x^\mu_k$, since, by homogeneity, $\tau_1(\bx_k^\mu)$ is in particular the Euclidean best approximation of $\mF$ in $\spann(\tau_1(\bx_k^\mu))$. (More concretely, it follows from choosing $x^\mu = y^\mu = \alsx^\mu$ in~\eqref{eq: normal equation} that $F(\bx_n^\mu) = \| \tau_1(\bx_n^\mu) \|_\frob$.)
\end{proof}

\begin{proposition}\label{prop:increasing norm of tau}
For $\nu \ge \mu$ and $\ell \ge k$ it holds
\[
\| \tau_1 (\bx_\ell^\nu) \|_\frob \ge \| \tau_1(\bx_k^\mu) \|_\frob.
\]
\end{proposition}
\begin{proof}
This is an immediate consequence of Proposition~\ref{prop:necessary optimality condition}, as, by the decreasing property of ALS, $\| \mF - \tau_1 (\bx_\ell^\nu) \|_\frob^2 \le \| \mF - \tau_1 (\bx_k^\mu) \|_\frob^2$. Alternatively, the statement follows from~\eqref{eq: norm of taukmu} and the monotonicity of HOPM.
\end{proof}

The next two key conclusions were drawn in~\cite[Lemma~4.1]{LiUschmajewZhang2013}. The first is as crucial as it is trivial.

\begin{proposition}\label{prop:norm monotonically increasing}
For every $\mu = 1,2,\dots,d$ the sequence $(\| x^\mu_k \|)$ of norms is monotonically increasing.
\end{proposition}
\begin{proof}
As in every inner step of Algorithm~\ref{Alg:ALS} only one block is updated, this follows from Proposition~\ref{prop:increasing norm of tau} and~\eqref{eq: product of norms}.
\end{proof}

As a result, the norms of the factors $x_k^\mu$ remain bounded from below \emph{and} from above. 


\begin{proposition}\label{prop:key observation}
For all $k \ge 1$ and $\mu = 1,2,\dots,d$ it holds
\[
0 < \| x^\mu_0 \| \le \| x^\mu_k \| \le \| \mF \|_\frob^{} \Bigg(\prod_{\nu \neq \mu} \| x^\nu_0 \|^{-1} \Bigg).
\]
\end{proposition}
\begin{proof}
Since $F^1(\bx_0) \neq 0$, we have $\| x^\mu_0 \| > 0$ for all $\mu \ge 2$. Then also $x_0^1 = x_1^1 \neq 0$ (the equality was assumed at the beginning of the section). The inequality $\| x^\mu_0 \| \le \| x^\mu_k \|$ holds by Proposition~\ref{prop:norm monotonically increasing}. Combining with~\eqref{eq: product of norms} and Proposition~\ref{prop:necessary optimality condition} gives
\[
\Bigg(\prod_{\nu \neq \mu} \| x^\nu_0 \| \Bigg) \| x^\mu_k \| \le \| \tau_1(\bx_k) \|_\frob^{} \le \| \mF \|_\frob^{},
\]
that is, the third inequality in the assertion.
\end{proof}

We now turn to the assumptions in Theorems~\ref{th:convergence via Lojasiewicz} and~\ref{th:convergence rate}.

\begin{proposition}\label{prop:function decrease per microstep}
In loop $k$, the decrease in function value of block update $\mu$ satisfies
\[
f(\bx_{k+1}^{\mu-1}) - f(\bx_{k+1}^\mu) =  \frac{\sigma_{k+1}^\mu}{2} \| x^\mu_{k+1} - x^\mu_k \|^2,
\]
where
\[
\sigma_{k+1}^\mu = \| x^1_{k+1} \|^2 \cdots \| x^{\mu-1}_{k+1} \|^2  \cdot  \| x^{\mu+1}_k \|^2 \cdots \| x^d_k \|^2.
\]
\end{proposition}
\begin{proof}
This is standard least squares theory: the update $x^\mu_{k+1}$ is chosen such that the gradient of the quadratic form $x^\mu \mapsto f(x^1_{k+1},\dots,x^{\mu-1}_{k+1},x^\mu,x^{\mu+1}_{k},\dots,x^d_{k})$ is zero. Its quadratic term is, using~\eqref{eq: product of norms},
\[
x_\mu \mapsto \frac{1}{2}\|\tau_1(x^1_{k+1},\dots,x^{\mu-1}_{k+1},x^\mu,x^{\mu+1}_{k},\dots,x^d_{k}\|_\frob^2 = \frac{\sigma_{k+1}^\mu}{2} \| x^\mu \|^2,
\]
hence the Hessian in every point is $\sigma^\mu_{k+1} I_{n_\mu}$. A Taylor expansion around $x^\mu_{k+1}$ proves the claim.
\end{proof}

\begin{proposition}\label{prop:decrease in function value}
The decrease in function value per outer loop satisfies
\[
f(\bx_k) - f(\bx_{k+1}) \ge \frac{\sigma_0}{2} \| \bx_{k+1} - \bx_{k} \|^2,
\]
where
\[
\sigma_0 = \min_{\mu=1,2,\dots,d} \sigma^\mu_1 > 0.
\]
\end{proposition}
\begin{proof}
By Proposition~\ref{prop:key observation}, we have $\sigma^\mu_k \ge \sigma^\mu_1 \ge \sigma_0 > 0$ for all $\mu = 1,2,\dots,d$. Building a telescopic sum, Proposition~\ref{prop:function decrease per microstep} yields
\[
f(\bx_k) - f(\bx_{k+1}) = \sum_{\mu=1}^d f(\bx^{\mu-1}_{k+1}) - f(\bx^{\mu}_{k+1}) \ge \frac{\sigma_0}{2} \sum_{\mu=1}^d \| x^\mu_{k+1} - x^\mu_k \|^2 = \frac{\sigma_0}{2} \| \bx_{k+1} - \bx_k \|^2,
\]
as asserted.
\end{proof}

\begin{proposition}\label{prop:minimum stepsize holds}
There exists $\kappa > 0$ such that~\eqref{eq:minimum stepsize safeguard} holds.
\end{proposition}
\begin{proof}
By Proposition~\ref{prop:key observation}, the iterates $\bx_k^\mu$ (so in particular the $\bx_k$) remain in some compact set $B$ for all $k$. Let $\nabla_\mu f (\bx)$ denote the partial block gradient at $\bx$ with respect to $x^\mu$. As $f$ is continuously differentiable on $B$, there exists $L > 0$ such that
\[\| \nabla_\mu f(\bx) - \nabla_\mu f(\by) \| \le \| \nabla f(\bx) - \nabla f(\by) \| \le L \| \bx - \by \|\] for all $\bx,\by \in B$. Since by construction of the iterates it holds $\nabla_\mu f(\bx_{k+1}^\mu) = 0$, we deduce
\begin{align*}
\| \nabla f(\bx_k) \|^2 = \sum_{\mu = 1}^d \| \nabla_\mu f(\bx_k) \|^2
&= \sum_{\mu = 1}^d \| \nabla_\mu f(\bx_{k+1}^\mu) - \nabla_\mu f(\bx_k) \|^2 \\
&\le L^2 \sum_{\mu=1}^d \| \bx_{k+1}^\mu - \bx_k^{}  \|^2 \\
&\le L^2 \sum_{\mu=1}^d \| \bx_{k+1} - \bx_k \|^2 = L^2 d \| \bx_{k+1} - \bx_k \|^2.
\end{align*}
The second inequality follows from the fact that $\bx_{k+1}^\mu$ shares the first $\mu$ blocks with $\bx_{k+1}$, and the last $d - \mu$ blocks with $\bx_k$. Hence~\eqref{eq:minimum stepsize safeguard} holds with $\kappa = 1 / (L\sqrt{d})$.
\end{proof}

In summary, we obtain our main result.

\begin{theorem}\label{th:main result}
The iterates $(\bx_k)$ generated by Algorithm~\ref{Alg:ALS} converge to a point $\bx_*$ with $\nabla f(\bx_*) = 0$, where $f$ is given by~\eqref{eq:considered function}. The convergence rate estimates~\eqref{eq:rate} in terms of the (a-priori unknown) exponent in the \L{}ojasiewicz gradient inequality~\eqref{eq:Lojasiewicz inequality} at $\bx_*$ apply. 
\end{theorem}
\begin{proof}
As stated in Proposition~\ref{prop:minimum stepsize holds}, relation~\eqref{eq:minimum stepsize safeguard} holds for all $k$ and some $\kappa > 0$. Proposition~\ref{prop:decrease in function value} then implies that both~\eqref{eq: minimum decrease} and~\eqref{eq: stationarity condition} also hold. The result is therefore an instance of Theorems~\ref{th:convergence via Lojasiewicz} and~\ref{th:convergence rate}.
\end{proof}

Without going into detail, we shall not conceal that the appearance of the tensor order $d$ in the constant $\kappa$ obtained in the proof of Proposition~\ref{prop:minimum stepsize holds} may ultimately deteriorate the convergence rate stated in Theorem~\ref{th:convergence rate} for growing $d$. This rate, however, is not explicitly available anyway. Generally speaking, a dependence on the dimensionality has to be taken into account when relying on black-box tools like Theorems~\ref{th:convergence via Lojasiewicz} and~\ref{th:convergence rate}.

Due to the equivalence of ALS and HOPM in the sense of Proposition~\ref{prop: equivalence ALS HOPM}, Theorem~\ref{th:main result} in particular states that the limit $\lambda_* = \| \tau(\bx_*) \|_\frob$ is a singular value of the tensor $\mF$. There is no guarantee that it is the maximum singular value $\lambda^*$. Of course, by the monotonicity of HOPM, we would have $\lambda_* = \lambda^*$, if the starting guess $\lambda_0 = F(\bx_0)/(\|x_0^1\| \| x_0^2 \| \cdots \| x_0^d \|)$ happened to be larger than the second largest critical value (singular value), but ensuring this seems comparably hard as finding $\lambda^*$ itself.

\section{On generalizations to compositions of strongly convex functions with multilinear maps}\label{sec: generalization}

In this section we take a second look at the main arguments used in Sec.~\ref{sec:application to ALS} from an abstract perspective, much in the spirit of~\cite{XuYin2013}. We explain why these arguments do not easily apply to general low-rank tensor optimization tasks by means of cyclic block coordinate descent (BCD), unless regularization is used.



A generic low-rank optimization problem is the following. One is given a function
\[
J \vcentcolon \R^{n_1\times n_2 \times \dots \times n_d} \to \R,
\]
and a multilinear map
\[
\tau \vcentcolon V^1 \times V^2 \times \cdots \times V^d \to \R^{n_1 \times n_2 \times \cdots \times n_d},
\]
where $V^1,V^2,\dots,V^d$ are finite-dimensional vector spaces. Now denoting the elements of $V^1 \times V^2 \times \dots \times V^d$ by $\bx = (x^1,x^2,\dots,x^d)$, the task is to minimize the function 
\begin{equation}\label{eq: abstract f}
f(\bx) = J(\tau(\bx)) + \frac{\sigma_*}{2} \sum_{\mu=1}^d \| x^\mu \|^2,
\end{equation}
where $\sigma_* \ge 0$ is a regularization parameter. 

The most common examples for $J$ are the squared Euclidean distance $\| \mF - \mX \|_\frob^2$ to a given tensor $\mF$, and the energy functional $\frac{1}{2} \langle \mathsf{A} \mX, \mX \rangle_\frob - \langle \mathcal{B}, \mX \rangle_\frob$ of a ``high-dimensional'' linear system of equations $\mathsf{A} \mX = \mathcal{B}$ with $\mathsf{A}$ being a symmetric positive definite linear operator on $\R^{n_1 \times n_2 \times \dots \times n_d}$. The map $\tau$, on the other hand, represents a low-rank tensor format of some fixed rank. A notable example is the rank-$r$ CP format, which for $d = 3$ reads
\begin{equation}\label{eq: CP}
\tau_r(\bA,\bB,\bC) = \sum_{i=1}^r a_i \tenprod b_i \tenprod c_i
\end{equation}
with $a_i$, $b_i$, and $c_i$ being the columns of $\bA \in \R^{n_1 \times r}$, $\bB \in \R^{n_2 \times r}$, and $\bC \in \R^{n_3 \times r}$, respectively. For $r=1$ we recover~\eqref{eq:rank-one tensors}. Other important examples are the Tucker format, the hierarchical Tucker format, and the tensor train format. We refer to~\cite{GrasedyckKressnerTobler2013,Hackbusch2012,KoldaBader2009} and references therein.

The generalization of Algorithm~\ref{Alg:ALS} to $f$ given by~\eqref{eq: abstract f} is the cyclic BCD method noted in Algorithm~\ref{Alg:cyclic BCD}. 
\begin{algorithm}
\DontPrintSemicolon
\KwIn{Starting guess $\bx_0$.}
$k \leftarrow 0$\;
\While{not converged}{
\For{$\mu = 1,2,\dots,d$}{
$\displaystyle x^\mu_{k+1} \in \argmin_{x^\mu \in V^\mu} f(x^1_{k+1},\dots,x^{\mu-1}_{k+1},x^\mu,x^{\mu+1}_k,\dots,x^d_k)$\;
}
$k \leftarrow k+1$\;
}
\caption{Cyclic BCD for low-rank optimization}\label{Alg:cyclic BCD}
\end{algorithm}
It is feasible whenever $J$ has bounded sub-level sets. We shall investigate to what extent one can prove convergence using the same ideas as in Sec.~\ref{sec:application to ALS}. To this end, we assume that $J$ is real-analytic, convex, and coercive, that is, $\| \mX \|_\frob \to \infty$ implies $J(\mX) \to \infty$. The two above-mentioned examples have this property. Then $f$ in~\eqref{eq: abstract f} is real-analytic, and at least the restriction to every block-variable is convex. For fixed $\mX_0 = \tau(\bx_0)$, let $\mathcal{L}_0 = \{ \mX \vcentcolon J(\mX) \le J(\mX_0) \}$. Letting $\gamma_0 \ge 0$ be a lower spectral bound for the Hessians $\nabla^2 J(\mX)$ on the compact convex set $\mathcal{L}_0$, it follows from Taylor's theorem that
\begin{equation}\label{eq:strong convexity of J}
J(\mY) \ge J(\mX) + \langle \nabla J(\mX), \mY - \mX \rangle_\frob + \frac{\gamma_0}{2} \| \mY - \mX \|_\frob^2
\end{equation}
for all $\mX, \mY$ in $\mathcal{L}_0$. 
Let us further introduce the quantities
\begin{equation}\label{eq:definition of mkmu}
\sigma^\mu_{k+1} = \min_{x^\mu \neq 0} \frac{\| \tau(x^1_{k+1},\dots,x^{\mu-1}_{k+1},x_{}^\mu,x^{\mu+1}_k,\dots,x^d_k) \|_\frob^2}{\|x^\mu\|^2},
\end{equation}
which are easily identified as the squared minimal singular values of the restricted linear maps $x^\mu \mapsto \tau(x^1_{k+1},\dots,x^{\mu-1}_{k+1},x_{}^\mu,x^{\mu+1}_k,\dots,x^d_k)$ that arise during the iteration. They will play a similar same role as the $\sigma_{k+1}^\mu$ in Proposition~\ref{prop:function decrease per microstep}. Note that if $\max(\gamma_0^{} \sigma_{k+1}^\mu,\sigma_*) > 0$, then the update $x^\mu_{k+1}$ is a unique choice, since in this case the corresponding restricted minimization problem is strongly convex (see~\eqref{eq: abstract function value decrease}).

The new entry $x^\mu_{k+1}$ satisfies $\nabla_\mu f(\bx_{k+1}^\mu) = 0$. Specifically, by the chain rule and the linearity of $\tau$ with respect to every block variable, this implies
\begin{align*}
0 &= \langle \nabla_\mu f(\bx_{k+1}^\mu),  x^\mu_k - x^\mu_{k+1} \rangle \\
&= \langle \nabla J(\tau(\bx_{k+1}^\mu)), \tau(\bx_{k+1}^{\mu-1}) - \tau(\bx_{k+1}^\mu)  \rangle_\frob^{} + \sigma_* \langle x^\mu_{k+1},x^\mu_k - x^\mu_{k+1} \rangle.
\end{align*}
Since all generated tensors $\tau(\bx_k^\mu)$ remain in $\mathcal{L}_0$, it then follows from~\eqref{eq:strong convexity of J} that
\begin{equation}\label{eq: abstract function value decrease}
 f(\bx_{k+1}^{\mu-1}) - f(\bx_{k+1}^\mu) \ge \frac{\gamma_0^{}  \sigma_{k+1}^\mu + \sigma_*}{2} \| x_{k+1}^\mu - x_k^\mu \|^2,
\end{equation}
which is the analog to Proposition~\ref{prop:function decrease per microstep}.

\subsection{The regularized case}

Suppose we have chosen $\sigma_* > 0$. Then we can easily deduce an analog of Proposition~\ref{prop:decrease in function value} from~\eqref{eq: abstract function value decrease}. Further, the sub-level sets of $f$ are bounded when $\sigma_* > 0$ (as $J$ is bounded below). Hence, since $f$ is decreasing, the sequences $(x_k^\mu)$ themselves remain bounded for every $\mu$, which in turn allows to prove an analog of Proposition~\ref{prop:minimum stepsize holds}. These two propositions were sufficient to prove Theorem~\ref{th:main result}, which therefore can be generalized as follows.

\begin{theorem}[{cf. Xu and Yin~\cite[Theorems~2.8 and 2.9]{XuYin2013}}]
Let $J$ be real-analytic, convex, and coercive, and $\tau$ be multilinear as considered above. A sequence $(\bx_k)$ of iterates generated by Algorithm~\ref{Alg:cyclic BCD} for the function $f$ given by~\eqref{eq: abstract f} with $\sigma_* > 0$ is uniquely determined by $\bx_0$, and converges to a point $\bx_*$ satisfying $\nabla f(\bx_*) = 0$. The convergence rate estimates~\eqref{eq:rate} apply correspondingly.  
\end{theorem}


\subsection{The non-regularized case}

When $\sigma_* = 0$, we need $\gamma_0 > 0$ (which is always possible if $f$ is strictly convex), but also have to assume that
\begin{equation}\label{eq:lower bound on mkmu}
\liminf_{k,\mu} \sigma_k^\mu = \sigma_0 > 0,
\end{equation}
in order to deduce an analog of Proposition~\ref{prop:decrease in function value} from~\eqref{eq: abstract function value decrease}. As the $\tau(\bx_k^\mu)$ remain bounded, property~\eqref{eq:lower bound on mkmu} then implies, in light of~\eqref{eq:definition of mkmu}, that the sequences $(x_k^\mu)$ also remain bounded for every $\mu$, so that an analog of Proposition~\ref{prop:minimum stepsize holds} again can be established.

\begin{theorem}\label{th:main result generalized}
Let $J$ be real-analytic, strictly convex, and coercive, and $\tau$ be multilinear as considered above. A sequence $(\bx_k)$ of iterates generated by Algorithm~\ref{Alg:cyclic BCD} for the function $f$ given by~\eqref{eq: abstract f} with $\sigma_* = 0$ satisfying~\eqref{eq:lower bound on mkmu} converges to a point $\bx_*$ satisfying $\nabla f(\bx_*) = 0$. The convergence rate estimates~\eqref{eq:rate} apply correspondingly.  
\end{theorem}

Condition~\eqref{eq:lower bound on mkmu} can be interpreted as a stability requirement on the used tensor format during the iteration. Unless one finds a good argument to guarantee it in advance (for instance some condition on the starting guess), Theorem~\ref{th:main result generalized} remains an a-posteriori statement of minor practical value. For Algorithm~\ref{Alg:ALS}, the product formula~\eqref{eq: product of norms} and Proposition~\ref{prop:key observation} (which itself is proved using~\eqref{eq: product of norms}) imply~\eqref{eq:lower bound on mkmu}. Unfortunately, a property like~\eqref{eq: product of norms} is a unique feature of rank-one tensors. For none of the aforementioned tensor formats involving notions of higher rank an argument ensuring the stability condition~\eqref{eq:lower bound on mkmu} is currently available. In the case of optimization using rank-$r$ CP tensors~\eqref{eq: CP} with $r > 1$, this may be explained by the fact that the problem itself can be ill-posed~\cite{
DeSilvaLim2008}. Another reason for~\eqref{eq:lower bound on mkmu} to fail can be that the used rank in the multilinear tensor format overestimates the actual rank of the sought solution.

\bibliography{HOPM}
\bibliographystyle{siam}

\end{document}